\documentclass[
final
 , nomarks
]{dmtcs-episciences}

\usepackage[utf8]{inputenc}
\usepackage{subfigure}

\usepackage{enumerate}

\usepackage{amsmath}
\usepackage{amssymb}

\usepackage[round,sort&compress]{natbib}

\newtheorem{theorem}{Theorem}[section]
\newtheorem{claim}{Claim}[section]
\newtheorem{lemma}[theorem]{Lemma}
\newtheorem{Proposition}[claim]{Proposition}
\newtheorem{remark}[theorem]{Remark}
\newtheorem{case}{Case}

\author[Jie Zhang, Zhilan Wang, Jin Yan]{Jie Zhang
  \and Zhilan Wang
  \and Jin Yan\thanks{The author's work is supported by NNSF of China (No.12071260).}}

\title{A characterization of rich $c$-partite ($c \geq 8$) tournaments without $(c+2)$-cycles}

\affiliation{School of Mathematics, Shandong University, Jinan, China}
\begin{document}

\publicationdata{vol. 25:2 }{2023}{19}{10.46298/dmtcs.9732}{2022-06-23; 2022-06-23; 2022-12-26; 2023-08-02}{2023-09-24}
\maketitle
\begin{abstract}
  Let $c$ be an integer. A $c$-partite tournament is an orientation of a complete $c$-partite graph. A $c$-partite tournament is rich if it is strong and each partite set has at least two vertices. In 1996, Guo and Volkmann characterized the structure of all rich $c$-partite tournaments without $(c+1)$-cycles, which solved a problem by Bondy. They also put forward a problem that what the structure of rich $c$-partite tournaments without $(c+k)$-cycles for some $k \geq 2$ is. In this paper, we answer the question of Guo and Volkmann for $k=2$.
\end{abstract}

\section{Introduction}

In this paper, we consider only finite digraphs without loops or multiple arcs. For a digraph $D$, we denote its vertex set by $V(D)$ and its arc set by $A(D)$. A digraph is \emph{strong} if, for every pair $x$, $y$ of distinct vertices in $D$, there exist a path from $x$ to $y$ and a path from $y$ to $x$. The notation $q$-cycle ($q$-path) means a cycle (path) with $q$ arcs. We will use $(A,B)$-arc to denote an arc from a vertex in $A$ to a vertex in $B$. A \emph{$c$-partite tournament} is an orientation of a complete $c$-partite graph and is \emph{rich} if it is strong and each partite set has at least two vertices. We denote by $\mathcal{D}$ the family of all rich $c$-partite ($c\geq 5$) tournaments. It is clear that \emph{tournaments} are special $c$-partite tournaments on $c$ vertices with exactly one vertex in each partite set.

An increasing interest is to generalize results in tournaments to larger classes of digraphs, such as multipartite tournaments. For results on tournaments and multipartite tournaments, we refer the readers to \cite{Bang-Jensen2,Bang-Jensen,Bang-Jensen3,Beineke,Volkmann}. Many researchers have done a lot of work on the study of cycles whose length do not exceed the number of partite sets. In 1976, \cite{bondy} proved that every strong $c$-partite ($c\geq 3$) tournament contains a $k$-cycle for all $k \in \{3, 4, \ldots, c\}$. He also showed that every $c$-partite tournament in $\mathcal{D}$ contains a $q$-cycle for some $q>c$, and asked the following question: does every multipartite tournament of $\mathcal{D}$ contains a $(c+1)$-cycle? A negative answer to this question was obtained by \cite{gutin1}. The same counterexample was found independently by \cite{Balakrishnan}. Further in \cite{gutin2}, the following result was proved.

\begin{theorem}\cite{gutin2} \label{Gutin} Every multipartite tournament in $\mathcal{D}$ has a $(c+1)$-cycle or a $(c+2)$-cycle. \end{theorem}

In \cite{guo}, $\mathcal{W}_m$ is defined as follows. Let $c\ (\geq 5)$ be an integer and $P=x_1\cdots x_m$ be a path with $m \geq c$. The $c$-partite tournament consisting of the vertex set $\{x_1,\ldots,x_m\}$ and the arc set $A(P)\cup \{x_ix_j: i-j>1\ \textup{and}\ i \not\equiv j (\textup{mod}\ c)\ \textup{where}\ i,j \in [m]\}$ is denoted by $W_m$. The set of all $c$-partite tournaments obtained from $W_m$ by replacing $x_i$ by a vertex set $A_i$ with $|A_i| \geq 2$ for $i\in \{1,2,m-1,m\}$ is denoted by $\mathcal{W}_m$.

In 1996, \cite{guo} gave a complete solution of this problem of Bondy and determined the structure of all $c$-partite ($c \geq 5$) tournaments of $\mathcal{D}$, that have no $(c+1)$-cycle.

\begin{theorem}\cite{guo}\label{guo2} Let $D$ be a c-partite tournament in $\mathcal{D}$. Then $D$ has no $(c+1)$-cycle if and only if $D$ is isomorphic to a member of $\mathcal{W}_m$. \end{theorem}


In this paper, we characterize all $c$-partite ($c\geq 8$) tournaments in $\mathcal{D}$ without $(c+2)$-cycles. Before defining families $\mathcal{Q}_m$ and $\mathcal{H}$, we present the main theorem.

\begin{theorem} \label{main} Let $D$ be a c-partite ($c\geq 8$) tournament in $\mathcal{D}$. Then $D$ has no $(c+2)$-cycle if and only if $D$ is isomorphic to a member of $\mathcal{Q}_m$ or $\mathcal{H}$. \end{theorem}


The families $\mathcal{Q}_m$ and $\mathcal{H}$ are described as follows.

$\bullet$ Let $i$ be a given integer with $2<i< c-1$. Define $\mathcal{H}^\prime$ the set of $(c+1)$-partite tournaments whose partite sets are $V_1,\ldots,V_{c+1}$, where $V_1=\{v_1\}$, $|V_i| \geq 1$ and $|V_j| \geq 2$ for $j \in [c+1]\setminus \{1,i\}$, and the arc set consists of arcs from each vertex of $V_{j_1} $ to each vertex of $ V_{j_2}$, where $2 \leq j_1 < j_2 \leq c+1$, and arcs between $v_1$ and vertices in other partite sets with arbitrary directions.
The family of all $c$-partite tournaments obtained from a member of $\mathcal{H}^\prime$ by deleting all arcs between $v_1$ and $V_i$ and merging $V_1$ and $V_i$ into a partite set is denoted by $\mathcal{H}$. 
\bigskip

\begin{figure}[h]
  \begin{center}
  \includegraphics[width=0.4\textwidth]{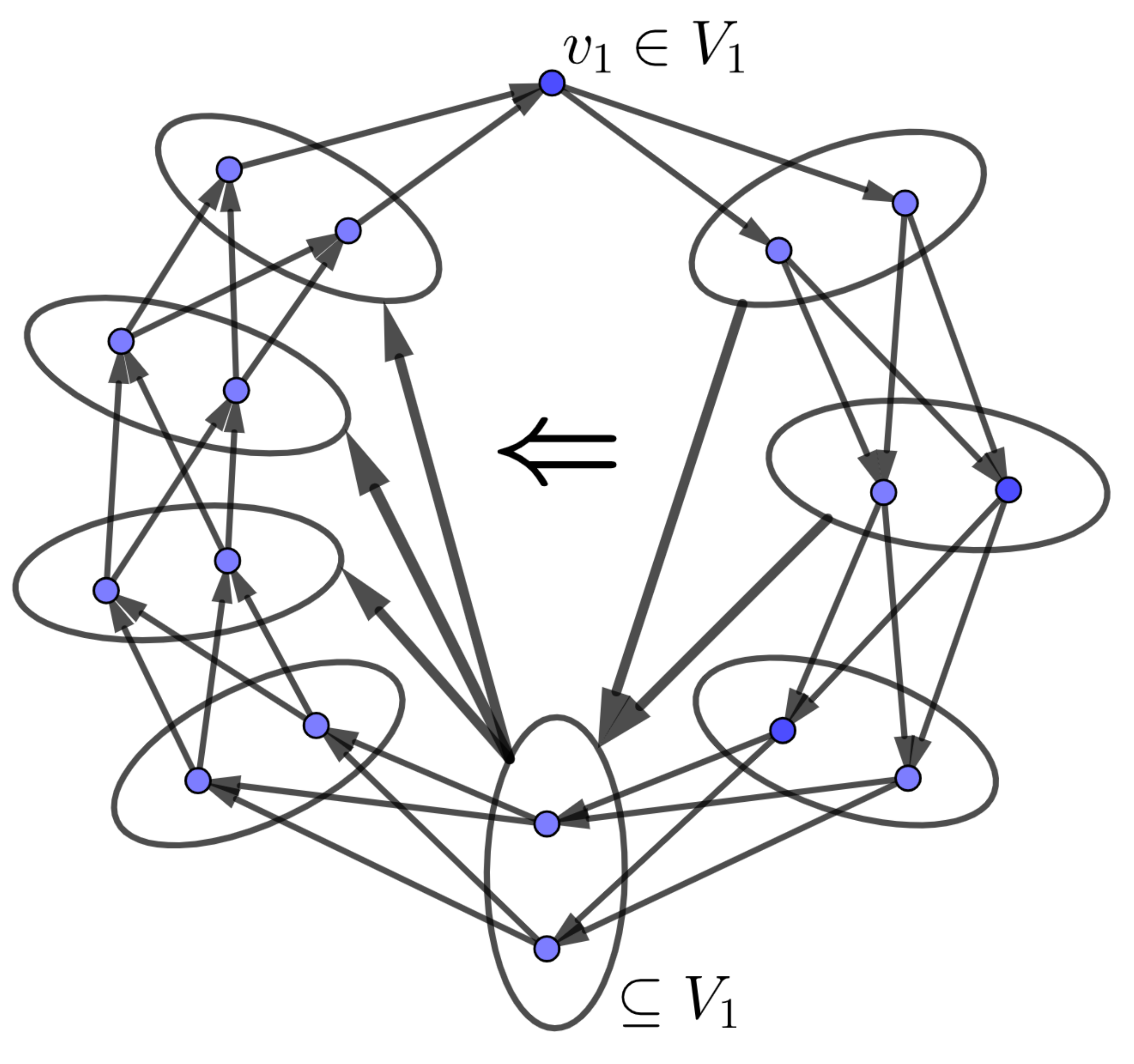}
  \caption{An example of the $8$-partite tournament with $|V_1|=3$ and $|V_j|=2$ for $2 \leq j \leq 8$. Here, the arcs between $v_1$ and other vertices are arbitrary.}
  \label{fig1}
  \end{center}
\end{figure}

\medskip

$\bullet$ Let $s$ and $t$ be two fixed integers with $1 \leq s <t-1 \leq c$ and $P=x_1\cdots x_m$ be a path with $m \geq c$. We denote $Q_m^\prime$ the $(c+1)$-partite tournament consisting of the vertex set $\{x_1,\ldots,x_m\}$ and the arc set $A(P)\cup \{x_ix_j: i-j>1\ \textup{and}\ i \not\equiv j (\textup{mod}\ c)\ \textup{where}\ i,j \in [m]\}$. Deleting arcs of $Q_m^\prime$ between $\{x_i|\ i\equiv s\ (\textup{mod}\ (c+1)) \}$ and $\{x_j|\ j\equiv t\ (\textup{mod}\ (c+1)) \}$ and  and merging $V_i$ and $V_j$ into a partite set, we obtain a $c$-partite tournament $Q_m^1$.

Let $\mathcal{Q}_m=\mathcal{Q}_m^1 \cup \mathcal{Q}_m^2 $, where $\mathcal{Q}_m^1$ and $ \mathcal{Q}_m^2 $ are defined as follows.
\begin{enumerate}[(\begin{math}\mathcal{Q}_m^1\end{math})]
\item The set of all $c$-partite tournaments obtained from $Q_m^1$ by substituting $x_i$ with a vertex set $A_i$ is denoted by $\mathcal{Q}_m^1$ for
    \begin{enumerate}[(1)]
      \item $i \in \{1,2,m-1,m\};$ or
      \item $i=t$ when $s=1$ and $t=3$ or $4$; or
      \item $i=m-2$ when $\{m,m-2\}\equiv\{s,t\}$ (mod $(c+1)$), or $i=m-3$ when $\{m,m-3\}\equiv\{s,t\}$ (mod $(c+1)$).
    \end{enumerate}

\item $\mathcal{Q}_m^2$ is the set of all $c$-partite tournaments obtained from a member of $\mathcal{Q}_m^1$ by reversing some arcs satisfying
\begin{align}
\begin{split}
\left\{
  \begin{array}{ll}
    (A_2, A_3)\hbox{-arcs},& \hbox{when $t=3,s=1$;}\\
    (A_1, A_2)\hbox{-arcs}, & \hbox{when $t=c+1,s=2$;} \\
    (A_{m-2}, A_{m-1})\hbox{-arcs}, & \hbox{when $\{m-2, m\} \equiv \{s,t\} \ (\textup{mod}\ (c+1)) $;} \\
    (A_{m-1}, A_m)\hbox{-arcs},  & \hbox{when $\{m-1, m-c\} \equiv \{s,t\} \ (\textup{mod}\ (c+1)) $.}  \nonumber
  \end{array}
\right.
\end{split}
\end{align}
%
\end{enumerate}

Note that, in our main theorem, the parameter $c$ is at least 8. This condition may be not sharp. The characterization of rich $c$-partite tournaments with $5 \leq c \leq 7$ needs more techniques. We conclude this section by giving the following organization. In the second section, we set up notation and some helpful lemmas. The proof of the main theorem is presented in the third section.

\bigskip
\smallskip

\begin{figure}[h]
  \begin{center}
  \includegraphics[width=0.8\textwidth]{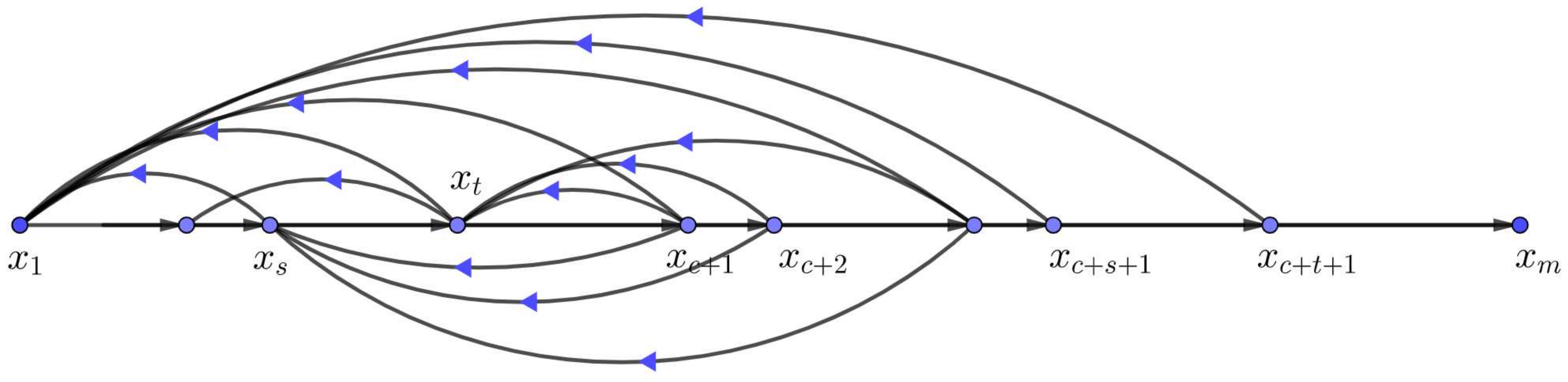}
  \caption{An example of $Q_m^1$. Here, all other possible arcs are of the same direction as the path. }
  \label{fig2}
  \end{center}
\end{figure}

\section{Notation and useful lemmas}

\medskip
\begin{flushleft}
{\large{\emph{2.1 Notation  }}}
\end{flushleft}

For terminology and notation not defined here, we refer to \cite{Bang-Jensen}. Let $D$ be a digraph. For the vertex $x \in V(D)$, the set of out-neighbours of $x$ is denoted by $N_D^+(x)=\{y \in V(D) |\ xy \in A(D)\}$ and the set of in-neighbours of $x$ is denoted by $N_D^-(x)=\{y \in V(D) |\ y x \in A(D)\}$, respectively. For a vertex set $X \subseteq V (D)$, we define $N^+(X)= N_D^+(X) = \cup_{x\in X} N_D^+(x)\setminus X$ and $N^-(X) = N_D^-(X) = \cup_{x\in X} N_D^-(x)\setminus X$. When $X$ is a subdigraph of $D$, we write $N^+(X)$ instead of $N^+(V(X))$. We define $D[X]$ as the subdigraph of $D$ induced by $X$, and let $D - X=D[V(D)\setminus X]$. Define $[t] = \{1,\ldots, t\}$ for simplicity.

Let $C$ be a cycle (or path). For a vertex $v$ of $C$, the successor and the predecessor of $v$ on $C$ are denoted by $v^+$ and $v^-$, respectively. We write the $i$-th successor and the $i$-th predecessor of $v$ on $C$ as $v^{i+}$ and $v^{i-}$, respectively. The notation $v_i C v_j$ means the subpath of $C$ from $v_i$ to $v_j$ along the orientation of $C$. The length of $C$ is the number of arcs of $C$. We say a vertex $x$ outside $C$ can be \emph{inserted} into $C$ if there is an in-neighbour of $x$ on $C$, say $v$, such that $v^+$ is an out-neighbour of $x$. 

If $xy$ is an arc in $A(D)$, then we write $x \rightarrow y$ and say that $x$ dominates $y$. If $X$ and $Y$ are two disjoint vertex sets of $D$, we use $X \rightarrow Y$ to denote that every vertex of $X$ dominates every vertex of $Y$, and define $A \Rightarrow B$ that there is no arc from a vertex in $B$ to a vertex in $A$. If $X$ or $Y$ consists of a single vertex, we omit the braces in all following notation. Correspondingly, $x\nrightarrow y$ expresses that $xy \notin A(D)$.

A path $P=x\cdots y$ is \emph{minimal} if no proper subset of $V(P)$ induces a subdigraph of $D$ which contains a path from $x$ to $y$.
For two vertices $x$ and $y$ in $D$, the distance from $x$ to $y$ in $D$, denoted by $dist(x, y)$, is the length of a shortest path from $x$ to $y$ in $D$. The \emph{diameter} of $D$, denoted by $diam(D)$, is the maximum distance between all pairs of its vertices.

\begin{flushleft}
{\large{\emph{2.2 Useful lemmas.  }}}
\end{flushleft}

We give the following results that are frequently used in the proof of Theorem \ref{main}.
\begin{theorem}\cite{guo}\label{guo1} Let $D$ be a strong $c$-partite tournament. If $D$ has a $k$-cycle containing vertices from exactly $l$ partite sets with $l<c$, then $D$ has a $t$-cycle for all $t$ satisfying $k \leq t \leq c+(k-l)$.  \end{theorem}

\begin{remark} \label{fact1} Let $C$ be a $k$-cycle in a digraph $D$. If $D$ contains no $(k+1)$-cycle, then no vertex can be inserted into $C$. \end{remark}


\begin{lemma} \label{lem+-} Let $D$ be a multipartite tournament in $\mathcal{D}$ and $C$ a $(c+1)$-cycle of $D$. Suppose that $D$ has no $(c+2)$-cycle and $D-C \subseteq N^+(C) \cap N^-(C)$. Then for every $y \in D-C$, there exists a vertex $x \in C$ such that $x$ and $y$ belong to the same partite set of $D$ and have the same in-neighbours and out-neighbours in $C$.\end{lemma}

\begin{proof}
Let $C=x_1x_2\cdots x_iy_1x_{i+1}\cdots x_cx_1$ a $(c+1)$-cycle of $D$, where $x_j \in V_j$ for $j \in [c]$ and $y_1 \in V_1$. Clearly, there exists a vertex $x \in C$ such that $x$ and $y$ belong to the same partite set of $D$. Assume that $x= x_j$, as the case $x=x_1$ and the case $x=y_1$ are similar.

Suppose that $j=1$. First suppose that $yx_{1}^-\in A(D)$. Then $y \Rightarrow x_{i+1}Cx_c$. Obviously, if $y \Rightarrow x_{i}$ then $y \Rightarrow x_{2}Cx_i$. Since $y \in N^+(C) \cap N^-(C)$, we have $x_{i} \Rightarrow y$. Thus $x_{i} \Rightarrow y \Rightarrow x_{i+1}$ and $y$ and $y_1$ have the same in-neighbours and out-neighbours in $C$. Second suppose that $x_{1}^+y\in A(D)$. Similarly, we obtain that $y$ and $y_1$ have the same in-neighbours and out-neighbours in $C$ again. Hence $x_{1}^-y, yx_{1}^+ \in A(D)$ and $y$ and $x_1$ have the same in-neighbours and out-neighbours in $C$.

Set $j\in [c]\setminus \{1\}$. If $yx_{j}^-\in A(D)$, then $y \Rightarrow C$ because $D$ has no $(c+2)$-cycle, which contradicts the assumption. Thus $x_{j}^-y\in A(D)$ for $j \in \{2,\ldots,c\}$. Similarly, we obtain that $y \rightarrow x_j^+$ for $j \in \{2,\ldots,c\}$. Thus $y$ and $x_j$ have the same in-neighbours and out-neighbours in $C$.
\end{proof}

\begin{lemma} \label{clm1} Let $D$ be a $c$-partite tournament in $\mathcal{D}$ and $\mathcal{C}$ the family of all $(c+1)$-cycles of $D$. Suppose that $D$ has no $(c+2)$-cycle. If $D-C \subseteq N^+(C) \cap N^-(C)$ for every $C \in \mathcal{C}$, then $D \in \mathcal{H}$. \end{lemma}

\begin{proof}
Since $D$ has no $(c+2)$-cycle, it follows by Theorem \ref{guo1} that each $(c+1)$-cycle of $D$ meets all partite sets of $D$. This implies that each $(c+1)$-cycle contains exactly two vertices from one partite set and one vertex from other each partite set. Let $C \in \mathcal{C}$ and assume that it contains two vertices of $V_1$, that is, $C=x_1x_2\cdots x_iy_1x_{i+1}\cdots x_cx_1$, where $x_j \in V_j$ for $j \in [c]$ and $y_1 \in V_1$.

Since every partite set of $D$ has at least two vertices, there exist $c-1$ vertices $y_2,\ldots,y_c$ such that $y_j \in V_j $ for $j\in \{2,\ldots,c\}$. By Lemma \ref{lem+-}, we have $x_j^- \rightarrow y_j \rightarrow x_j^+$ for $j \in [c]\setminus \{1\}$. Note that $x_i$ and $y_i$ have the same in-neighbours and out-neighbours in $C$. Since $y_i$ is any vertex that is distinct with $x_i$ in $V_i$, each vertex in $V_i$ has the same in- or out-neighbors with $x_i$ for $i=2,\ldots ,c$. In the following, we often use this property to determine the direction of the arcs in $A(D)$. We get $x_1\rightarrow V_2 \rightarrow \cdots \rightarrow V_i \rightarrow y_1 \rightarrow V_{i+1} \rightarrow \cdots \rightarrow V_c \rightarrow x_1$. Let $C^\prime$ be the $(c+1)$-cycle $x_1y_2\cdots y_iy_1y_{i+1}\cdots y_cx_1$.

\begin{claim} \label{lem} The following statements hold.
\begin{align}
& (1)\ \{V_2,\ldots,V_{j-1}\} \rightarrow V_j \rightarrow \{V_{j+1},\ldots,V_{i}\}\ \text{ for }\ 2 \leq j \leq i-1;\nonumber \\
& (2)\ \{V_{i+1},\ldots,V_{j-1}\} \rightarrow V_j \rightarrow \{V_{j+1}, \ldots, V_c\} \ \text{ for }\ i+1 \leq j \leq c-1.\nonumber
\end{align}
\end{claim}
\begin{proof}
Suppose that there exists an integer $t \in \{2,\ldots,c\}\backslash \{2,i,i+1,c\}$ such that $y_{t+1}\rightarrow y_{t-1}$. If $x_t^{2+}\rightarrow y_t$, then there is a 6-cycle $x_t^{2+}y_ty_{t+1}y_{t-1}x_tx_{t}^+x_{t}^{2+}$ containing vertices from exactly four partite sets. If $x_t\rightarrow x_{t}^{2-}$, then $x_{t}x_{t}^{2-}x_{t}^-y_ty_{t}^+y_{t}^-x_t$ is a 6-cycle containing vertices from exactly four partite sets. In both cases, we deduce from Theorem \ref{guo1} that $D$ contains a $(c+2)$-cycle, a contradiction. This implies that $y_t \rightarrow x_t^{2+}$ and $x_t^{2-} \rightarrow x_t$. Then $y_tx_{t}^{2+}Cx_{t}^{2-}x_ty_{t+1}y_{t-1}y_t$ is a $(c+2)$-cycle, a contradiction. Thus we obtain that $y_{t-1}\rightarrow y_{t+1}$ for $t \in \{2,\ldots,c\}\backslash \{2,i,i+1,c\}$. By symmetry, we have $V_{t-1} \rightarrow V_{t+1}$. Since $c \geq 8$, it is easy to obtain that $x_j^{3-} \rightarrow y_j$ for $5 \leq j \leq i-1$ and $i+4 \leq j \leq c$ and $y_j \rightarrow x_j^{3+}$ for $ 2 \leq j\leq i-3$ and $i+1 \leq j \leq c-3$. We continue in this fashion to obtain $\{x_{j-1},\ldots,x_2\} \rightarrow y_j \rightarrow \{x_{j+1},\ldots,x_i\}$ for $j\in \{2,\ldots,i-1\}$ and $\{x_{j-1},\ldots,x_{i+1}\} \rightarrow y_j \rightarrow \{x_{j+1}, \ldots, x_c\}$ for $j \in \{i+1,\ldots,c-1\}$, successively. This proves Claim \ref{lem}. \end{proof}

By Claim \ref{lem}, we can obtain a $(c+2)$-cycle from a cycle with larger length. Now consider the arcs between $V_i$ and $V_{i+1}$. If $x_ix_{i+1}, x_{i+1}y_i \in A(D)$, then $D$ has a $(c+3)$-cycle $x_ix_{i+1}y_iy_1y_{i+1}x_{i+2}Cx_i$. If $x_{i+1}x_i, x_iy_{i+1} \in A(D)$, then there is a $(c+3)$-cycle $x_{i+1}x_iy_{i+1}C^\prime y_1x_{i+1}$. By $c \geq 8$ and Claim \ref{lem}, we can obtain a $(c+2)$-cycle from such a $(c+3)$-cycle. Thus $V_i \rightarrow V_{i+1}$ or $V_{i+1} \rightarrow V_i$.

Suppose that $V_i \rightarrow V_{i+1}$. We show that $\{V_2,\ldots,V_{i-1}\} \rightarrow \{V_i,\ldots,V_c\}$. If $x_{i+2}x_i \in A(D)$, then $x_{i+2}x_ix_{i+1}y_{i+2}C^\prime y_{i+1}x_{i+2}$ is a $(c+4)$-cycle. We can obtain a $(c+2)$-cycle because of $c \geq 8$ and Claim \ref{lem}. Based on this, considering arcs between $x_{i+3},\ldots,x_c$ and $x_{i}$ in order, we get $x_i \rightarrow \{x_{i+3},\ldots,x_c\} $. Similarly, it is immediate that $\{x_2,\ldots,x_{i-1}\} \rightarrow \{x_{i+1},\ldots,x_c\}$. Thus $\{V_2,\ldots,V_i\} \rightarrow \{V_{i+1},\ldots,V_c\}$.

For $V_{i+1} \rightarrow V_i$, we will get $ \{V_{i+1},\ldots,V_c\} \rightarrow \{V_2,\ldots, V_i\} $ in the same way. Note that the structures obtained in two cases are isomorphic, so we only consider the first structure in the following.

We declare that
\begin{equation}\label{yy}
\{V_2,\ldots,V_i \} \rightarrow y_1\rightarrow \{V_{i+1},\ldots,V_c\}.
\end{equation}
If $y_1 x_j \in A(D)$ for some $j \in [i-1]$, then $D$ contains a $(c+2)$-cycle $x_1C^\prime y_1x_jx_{i+1}C x_1$, a contradiction. If $x_j y_1 \in A(D)$ for some $j \in \{i+2,\ldots,c\}$, then $(c+2)$-cycle $x_jy_1C^\prime x_1C x_ix_j$ is in $D$. Thus (\ref{yy}) holds.

If $|V_1|=2$, $D$ is a member of $\mathcal{H}$, which proves this lemma. Thus assume that $|V_1| \geq 3$. We show that every vertex in $V_1\setminus \{x_1,y_1\}$ have the same in-neighbours and out-neighbours in $C$ as $y_1$. To see this, let $z_1$ be a vertex in $V_1\setminus \{x_1,y_1\}$. Suppose that $z_1\rightarrow x_i$. It is easy to see $z_1 \rightarrow x_2, \ldots,x_{i-1}$. If $x_j \rightarrow z_1$ for some $j \in \{i+1, \ldots, c\}$, then $x_t \rightarrow z_1$ for all $t \in \{j+1, \ldots, c\}$. Recall that $V(D-C) \subseteq N^+(C) \cap N^-(C)$, we have $x_c \rightarrow z_1$. Observe that there is a 6-cycle $z_1x_2y_cx_1y_2x_cz_1$ which meets 3 partite sets of $D$, a contradiction by Theorem \ref{guo1} again. Thus $x_i \rightarrow z_1$.
Obviously, $\{x_2,\ldots,x_{i-1}\} \rightarrow z_1$. Otherwise there exists a $(j+4)$-cycle $x_iz_1x_jx_cx_1C^\prime y_jx_i$ which meets $j+2$ partite sets of $D$ for $j \in\{2,\ldots,i-1\}$, a contradiction. On the other hand, it is easy to see $z_1 \rightarrow \{x_{i+1}\ldots,x_j\}$ if $z_1x_{j+1} \in A(D)$ for some $j \in \{i+1,\ldots,c\}$. Since $z_1 \in N^+(C) \cap N^-(C)$, we have $z_1x_{i+1} \in A(D)$. Thus $x_1x_2\cdots x_iz_1x_{i+1}\cdots x_cx_1$ is also a $(c+1)$-cycle. This implies that $z_1$ and $y_1$ have the same in-neighbours and out-neighbours in $C$. Hence $D$ is a member of $\mathcal{H}$. We are done.
\end{proof}

\section{Proof of Theorem \ref{main}}
Now we are ready to prove our main theorem. It is easy to see that every element of $\mathcal{H}$ and $\mathcal{Q}_m$ has no $(c+2)$-cycle. Hence, it suffices to show the converse is true as well. 

Suppose that $D$ is a $c$-partite tournament in $\mathcal{D}$ such that $D$ has no $(c+2)$-cycle and is not isomorphic to any element of $\mathcal{H}$ and $\mathcal{Q}_m$. Let $V_1,\ldots ,V_c$ be partite sets of $D$. By Theorem \ref{Gutin}, we know that $D$ contains a $(c+1)$-cycle. It follows by Theorem \ref{guo1} that each $(c+1)$-cycle of $D$ visits exactly one partite set twice and each other partite sets once. Let $\mathcal{C}$ be the set of all $(c+1)$-cycles of $D$.
Lemma \ref{clm1} gives that for every $C \in \mathcal{C}$, if all vertices of $D-C$ are contained in $N^+(C) \cap N^-(C)$, then $D \in \mathcal{H}$. Thus there exists at least one cycle $C$ in $\mathcal{C}$ such that $D-C$ contains a vertex outside $N^+(C) \cap N^-(C)$. Denote $C=x_1x_2\cdots x_{c+1}x_1$, where $x_j \in V_j$ for $j \in [i-1]$, $x_{j} \in V_{j-1}$ for $j \in \{i+1,\ldots,c+1\}$ and $x_i \in V_1$. Without loss of generality, assume that there exists a vertex $z \notin N^-(C)$. Because $D$ is strong, there is a path from $z$ to $C$. Let $P=z_1z_2 \cdots z_p$ be such a minimal path with $z_1 = z$ and assume that $z_p=x_t$. It is clear that, $p \geq 3$ and $z_2,\ldots,z_{p-2} \notin N^-(C)$, particularly, $z_{p-2} \nrightarrow x_{t-2}$ and $x_{t-1} \rightarrow z_{p-2}$. Since $D$ has no $(c+2)$-cycle, we see that $x_{t-2}z_{p-2} \notin A(D)$. This implies that there is no arc between $z_{p-2}$ and $x_{t-2}$, that is $x_{t-2}$ and $z_{p-2}$ must belong to the same partite set of $D$. It is not hard to get $z_{p-1}\nrightarrow  x_{t-1}$. Since, otherwise, $x_{t-3}$ and $z_{p-2}$ must belong to the same partite set of $D$, which is impossible. Together with $x_{t-1}\nrightarrow z_{p-1}$ we obtain that $x_{t-1}$ and $z_{p-1}$ belong to the same partite set of $D$. Further, vertices $x_{t-i}$ and $z_{p-i}$ belong to the same partite set for $1 \leq i \leq p-1$. It is obvious that $x_{t-2}\rightarrow z_{p-1}$ due to $V(C)\setminus x_t \Rightarrow z_{p-1}$.

We may assume that $x_t$ is on the path $x_{i+1}Cx_1$. If $C$ has a path from $x_t$ to $x_{t-2}$ with at most $c-1$ vertices, then together with the path $x_{t-2}x_{t-1}z_{p-2}z_{p-1}x_t$, we can form the path into a cycle of length at most $c+2$, which contains the vertices $x_{t-1}$ and $z_{p-1}$, and $x_{t-2}$, $z_{p-2}$ in the same partite sets respectively. We deduce that $D$ has a $(c+2)$-cycle from Theorem \ref{guo1}, a contradiction. This gives $x_{j+2}C x_{t-2} \Rightarrow x_j \Rightarrow x_tCx_{j-2}$ for $x_j \in x_tC x_{t-2}$. For the same reason, $x_{t-1} \Rightarrow x_iCx_{t-3}$ when $t > i+2$; and $ x_{t-1} \Rightarrow x_1Cx_{t-3}$ when $t = i+1$ or $i+2$.

\begin{claim} \label{clmc+2} $diam(D) \geq c+2$. \end{claim}

\begin{proof}
Since $x_{t-2}z_{p-1} \in A(D)$, it is not hard to obtain that $x_{t-1}$ dominates each vertex of $x_{t+1}Cx_{i-1}$ when $t \geq i+2$. Otherwise, there exists a vertex $x_j$ in $x_{t+1}Cx_{i-1}$ with $x_{t-1}\rightarrow x_{j+1}Cx_i$ and $x_j \rightarrow x_{t-1}$. Observe that $x_jx_{t-1}x_{j+1}Cx_{t-2}z_{p-1}x_tCx_j$ is a $(c+2)$-cycle, a contradiction. Therefore, $D[C]$ is isomorphic to $\mathcal{Q}_{c+1}$ with the initial vertex $x_t$ and the terminal vertex $x_{t-1}$. This implies that every minimal path from $z$ to $C$ must end at $x_t$ and $dist(z,x_{t-1}) \geq c+2$. Thus $diam(D) \geq c+2$ when $t \geq i+2$. If there is a vertex $x \notin V(C)\cup N^-(C)$ such that the minimal path from $x$ to $C$ which ends at the path $x_{i+2}Cx_1$, we complete the proof. Then all such minimal paths from $x$ to $C$ end at $x_{i+1}$, that is $x_t=x_{i+1}$. The following proof is divided into two cases.

\begin{case} $i \geq 5$; or $i = 4$ and $x_i \rightarrow x_{c+1}$. \end{case}

If $x_j \rightarrow x_i$ for $i+3\leq j \leq c$ and $i \geq 4$, or $x_{c+1} \rightarrow x_i$ when $i \geq 5$, observe that $x_jx_ix_2x_3x_1z_{p-2}$ $z_{p-1}x_{i+1}Cx_j$ is a cycle of length at most $c+2$ which visits $V_1$ three times, a contradiction. Thus in this case we have $x_i \rightarrow x_j$ for $i+3\leq j \leq c+1$. Hence $dist(z,x_{t-1}) \geq c+2$, which implies that $diam(D) \geq c+2$.

\begin{case} $i = 4$ and $x_{c+1} \rightarrow x_i$; or $i=3$. \end{case}

Recall that every partite set of $D$ has at least two vertices. Hence there is at least one vertex $y$ in $V_c \backslash \{x_{c+1}\}$. If $y \in N^+(C) \backslash N^-(C)$, then we choose $y$ as the vertex $z$, that is $y=z$. It is easy to check that $dist(y,x_1) \geq c+2$. If $y \in N^-(C)\backslash N^+(C)$, by considering the digraph $D^\prime $ obtained by reversing all arcs of $D$, we get $diam(D^\prime) \geq c+2$, that is $diam(D) \geq c+2$. If $y \in N^+(C)\cap N^-(C)$, then $x_c \rightarrow y \rightarrow x_1$ by Lemma \ref{lem+-}. This implies that $D$ contains a $(c+2)$-cycle $x_c y x_1 Cx_{i-1}x_{c+1}z_{p-1}x_{i+1}Cx_c$, a contradiction. Thus $diam(D) \geq c+2$ when $i \in \{3, 4\}$. This completes the proof of the claim.\end{proof}

Let $P=x_1x_2\cdots x_m$ be a path of $D$ with $dist(x_1,x_m)=diam(D)=m-1 \geq c+2$. As $D$ contains no $(c+2)$-cycle, vertices $x_i$ and $x_{c+i+1}$ must belong to the same partite set. If there exists vertex set $\{x_{i_1},x_{j_1},x_{i_2},x_{j_2}\} \subset V(D)$ with max$\{i_1,j_1,i_2,j_2\}-\mbox{min}\{i_1,j_1,i_2,j_2\} \leq c$ such that $x_{i_1}$ and $x_{j_1}$ belong to the same partite set and $x_{i_2}$ and $x_{j_2}$ belong to the same another partite set, then $D$ contains a $(c+2)$-cycle by applying Theorem \ref{guo1}, a contradiction. Thus $x_1Px_{c+1}$ meets all partite sets of $D$ and contains two vertices of exactly one partite set. Therefore, $D[P]$ is isomorphic to $Q_{m}$ with the initial vertex $x_1$ and the terminal vertex $x_m$. If $|V(D)|=m$, we are done. So $|V(D)|> m$. Assume that $x_j \in V_j$ for $j \in [i-1]$, $x_{j} \in V_{j-1}$ for $j \in \{i+1,\ldots,c\}$ and $x_i \in V_t$ for some $t \in [i-1]$. Let $x$ be a vertex of $D-P$. Suppose that $x \in N^+(P) \cap N^-(P)$, we now consider the arcs between $x$ and $V(P)$. We use $V_m$ to indicate the partite set which $x_m$ belongs to.

%

\begin{claim} \label{clm2} Suppose that $x \in N^+(P) \cap N^-(P)$. If there exist two vertices $x_p$ and $x_q$ on $P$ with $p < q$ such that $x_p \rightarrow x \rightarrow x_q$, then $x$ belongs to one of $\{V_1,V_2,V_{m-1},V_m\}$. Moreover, $x$ has the same in-neighbours and out-neighbours on $P$ as $x_l \in \{x_2,x_3,x_4,x_{m-3},$ $x_{m-2},x_{m-1}\}$, where
\begin{align}
\begin{split}
\left\{
  \begin{array}{ll}
    x_3 \in V_1, & \hbox{when $l=3$;} \\
    x_4 \in V_1, & \hbox{when $l=4$;} \\
    x_{m-2} \in V_m, & \hbox{when $l=m-2$;} \\
    x_{m-3} \in V_m, & \hbox{when $l=m-3$.} \nonumber
  \end{array}
\right.
\end{split}
\end{align}
\end{claim}
\begin{proof}
Since $D$ has no $(c+2)$-cycle, there is an integer $l$ such that $x_{l-1} \rightarrow x \rightarrow x_{l+1}$ and $x$ is in the same partite set with $x_l$. If $3 \leq l \leq m-2$, it easy to check that $D[P\cup \{x\}]$ contains a $(c+2)$-cycle $C$ as follows:\\
(1) when $m \geq c+l-1$,
\begin{itemize}
  \item $C=x_{l-1}xx_{l+1}Px_{c+l-2}x_lx_{l-2}x_{l-1}$, or
  \item $C=x_{l-1}xx_{l+1}Px_{c+l-3}x_lx_{l-3}x_{l-2}x_{l-1}$ unless $x_3 \in V_1$ and $l=3$, or
  \item $C=x_{l-1}xx_{l+1}Px_{c+l-4}x_lx_{l-4}Px_{l-1}$ unless  $l=3$, or $ x_4 \in V_1$ and $l=4$;
\end{itemize}
%
%
or (2) when $l \geq c-1$,
\begin{itemize}
  \item $C=x_{l-1}xx_{l+1}x_{l+2}x_lx_{l+2-c}Px_{l-1}$, or
  \item $C=x_{l-1}xx_{l+1}x_{l+2}x_{l+3}x_lx_{l+3-c}Px_{l-1}$ unless $x_{m-2} \in V_m$ and $l=m-2$, or
  \item $C=x_{l-1}xx_{l+1}Px_{l+4}x_lx_{l+4-c}Px_{l-1}$ unless $l=m-2$, or $ x_{m-3} \in V_m$ and $l=m-3$.
\end{itemize}
%
%
Then we consider $m < c+l-1$ and $l < c-1$. Recall that $m \geq c+3$. This implies that $l \geq 4$. Clearly, if $x_lx_1, x_{c+1}x_l \in A(D)$, there is a $(c+2)$-cycle $x_1Px_{l-1}xx_{l+1}Px_{c+1}x_lx_1$. If $x_l, x_{c+1} \in V_l$, then $D[P\cup \{x\}]$ contains a $(c+2)$-cycle $x_2Px_{l-1}xx_{l+1}Px_{c+2}x_lx_2$. Thus $x_l \in V_1$. Observe that $D[P\cup \{x\}]$ contains $x_3Px_{l-1}xx_{l+1}Px_{c+3}x_lx_3$ which is a $(c+2)$-cycle unless $l=4$. Thus $x$ belongs to $V_1,\ V_2,\ V_{m-1}$ or $V_m$. We also get $l\in \{2,3,4,m-3,m-2,m-1\}$ and $x_3 \in V_1$ when $l=3$; $x_4 \in V_1$ when $l=4$; $x_{m-2} \in V_m$ when $l=m-2$; and $ x_{m-3} \in V_m$ when $l=m-3$.

In all cases, it is easy to check that $x$ and $x_l$ have the same in-neighbours and out-neighbours on $P$, otherwise $D[P\cup \{x\}]$ contains a cycle of length at most $(c+2)$ and two pairs of vertices which belong to the same partite set. 
By Theorem \ref{guo1}, we get a contradiction.\end{proof}

\bigskip

\begin{figure}[h]
  \begin{center}
  \includegraphics[width=0.7\textwidth]{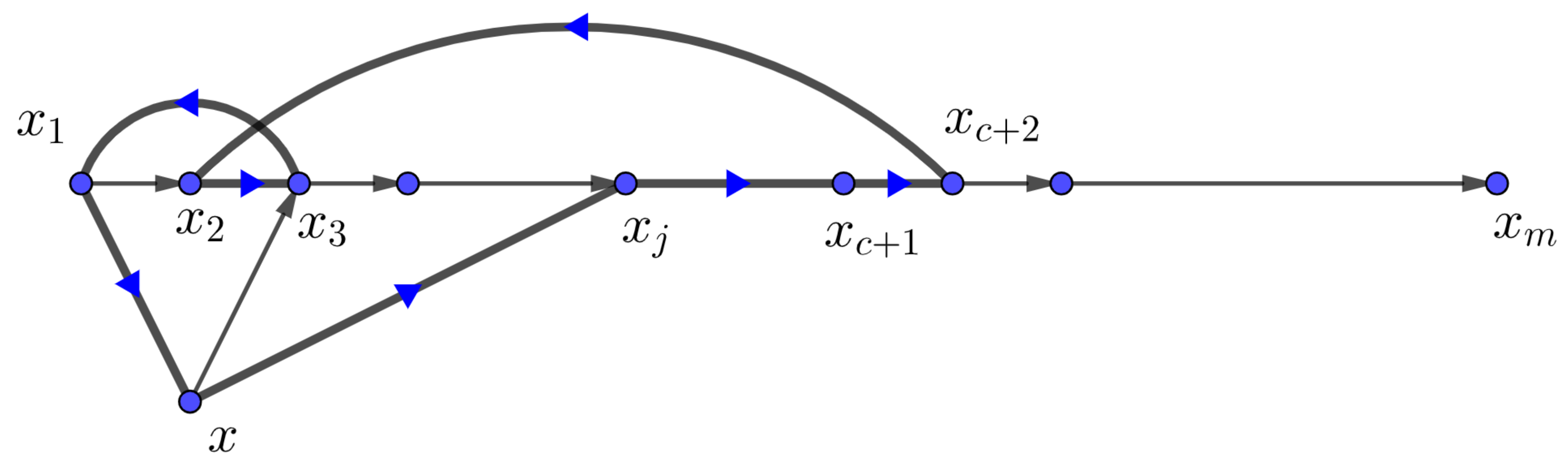}
  \caption{A cycle of length at most $(c+2)$ in $D[P\cup \{x\}]$ which contains two vertices in $V_1$ and two vertices in $V_2$.}
  \end{center}
  \label{example}
\end{figure}


\begin{claim} \label{clm3} Suppose that $x \in N^+(P) \cap N^-(P)$. If all vertices $x_p$ and $x_q$ with $x_p \rightarrow x \rightarrow x_q$ satisfy $p >q$, then\\
(i) $x$ has the same in-neighbours and out-neighbours on $P$ as $x_1$ or $x_m$; or\\
(ii) $D[P\cup\{x\}]$ has four specific structures as described in Fig. \ref{clm3-1}; or\\
(iii) $x\rightarrow x_1$, $x_2Px_m \Rightarrow x$ and $x \in V_{c+1}$; or\\
(iv) $x_m\rightarrow x$, $x \Rightarrow x_1Px_{m-1}$ and $x \in V_{m-c}$. \end{claim}

\begin{proof}
Note that if some vertex $x_q \rightarrow x$ (or $x \rightarrow x_q$) then $ x_qPx_m \Rightarrow x$ (or $x \Rightarrow x_1Px_q$). Let $q$ be the maximum integer such that $x \rightarrow x_q$.

First we suppose that $x$ has at least two in-neighbours and two out-neighbours on $P$. Let $x_{q_2}$ be the previous out-neighbour of $x$ before $x_q$ on $P$ and let $x_{p_1},x_{p_2}$ be two in-neighbours of $x$ on $P$ which is nearest to $x_q$. Clearly, we have $p_2-q_2<5$. Otherwise, $xx_{q_2}Px_{p_2}x$ is a 7-cycle meeting five partite sets of $D$, a contradiction by Theorem \ref{guo1}. Hence there exists at most one vertex in $x_{q_2}Px_{p_2}$ such that it is non-adjacent to $x$. Let $x_l$ be such vertex if it exists. According to the position of $x_l$, there are four possible sequences of $x_{q_2}Px_{p_2}$: (1) $x_{q_2}x_qx_{p_1}x_{p_2}$, (2) $x_{q_2}x_lx_qx_{p_1}x_{p_2}$, (3) $x_{q_2}x_qx_lx_{p_1}x_{p_2}$ and (4) $x_{q_2}x_qx_{p_1}x_lx_{p_2}$.

If $m \geq c+q$, then there is the $(c+2)$-cycle $xx_qPx_{c+q}x$ for sequences (1), (3) and (4) and the $(c+2)$-cycle $xx_{q_2}Px_{q_2+c}x$ for sequence (2) unless $x$ and $x_{q+c-2}$ are in the same partite set. Observe that for sequence (2) if $q \geq 5$, then there still exists a $(c+2)$-cycle $xx_{q-5}P x_{q-5+c}x$ via $x_{q+c-2}x \notin A(D)$. If $q \geq c$, then $xx_{q-c+1}Px_{p_1}x$ is a $(c+2)$-cycle for sequences (1)-(3) and $xx_{q-c+3}Px_{p_2}x$ is a $(c+2)$-cycle for sequence (4) unless $x$ and $x_{q-c+3}$ belong to the same partite set. We also note that for sequence (4) if $q \geq m-5$ and $q \geq c$, then there still exists a $(c+2)$-cycle $xx_{q-c+5}P x_{q+5}x$. Hence, $m < c+q$ and $q < c$ or $P$ meets the partite sets of $D$ along two special orders as described in Fig. \ref{clm3-1}. Moreover, $x$ has exactly two in-neighbours or two out-neighbours and $x \in \{V_1,V_2,V_m,V_{m-1}\}$.
\bigskip
\begin{figure}[h]
  \centering
  \subfigure[Type I with $q \leq 4$ and $m \geq c+q$.]{
  \includegraphics[width=0.45\textwidth]{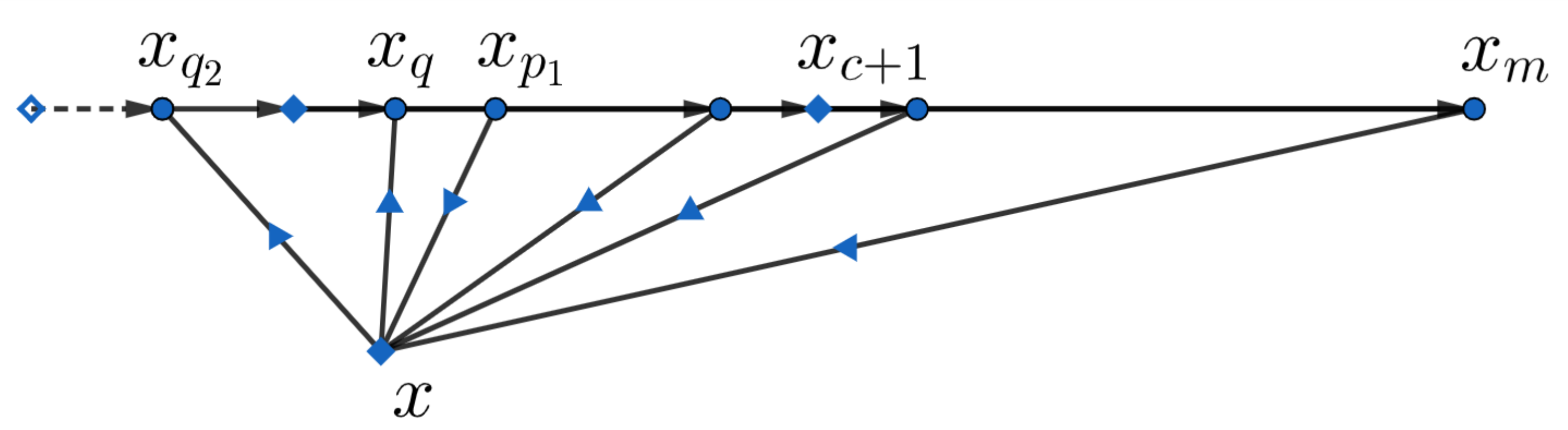}}
\qquad
  \subfigure[Type II with $q \geq$ max $\{c,m-5\}$.]{
  \includegraphics[width=0.45\textwidth]{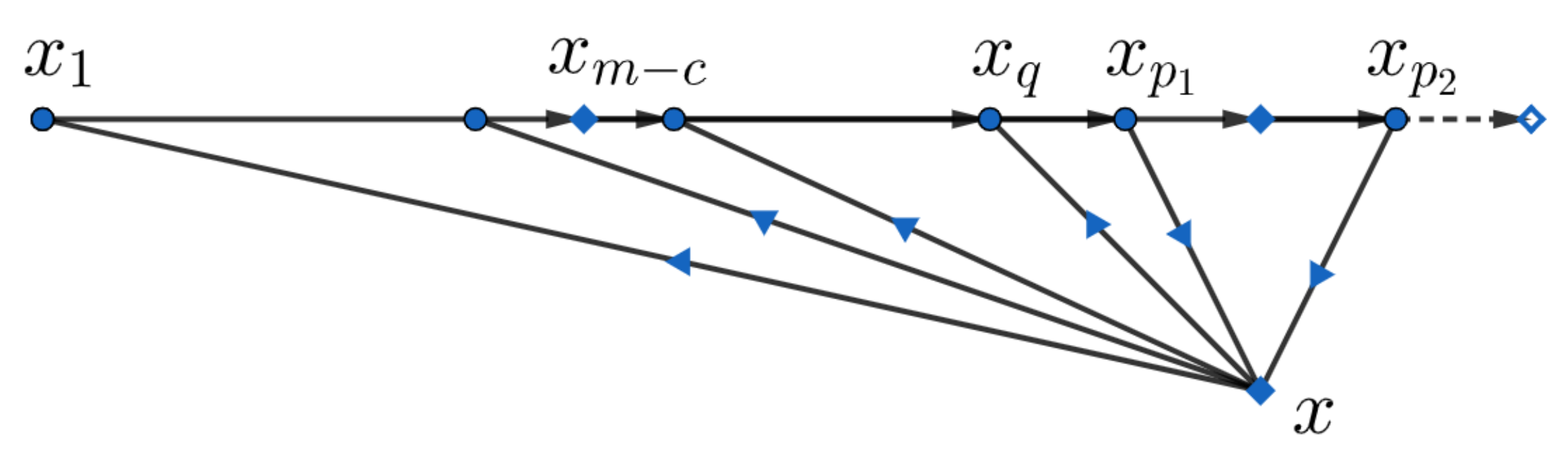}}
  \qquad
\bigskip
\bigskip

  \subfigure[Type III]{
  \includegraphics[width=0.45\textwidth]{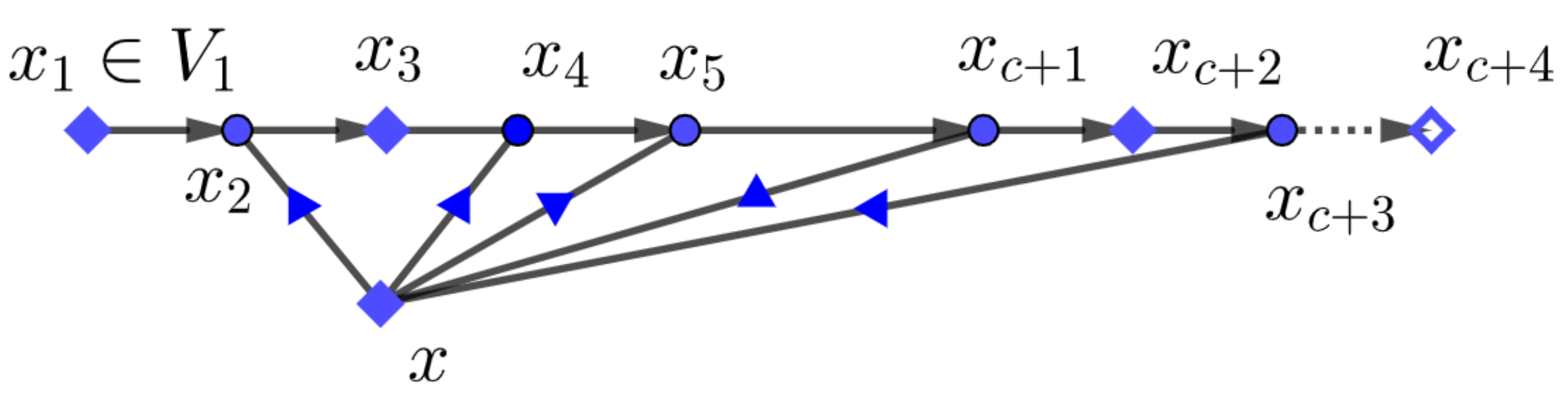}}
\qquad
\qquad
  \subfigure[Type IV]{
  \includegraphics[width=0.4\textwidth]{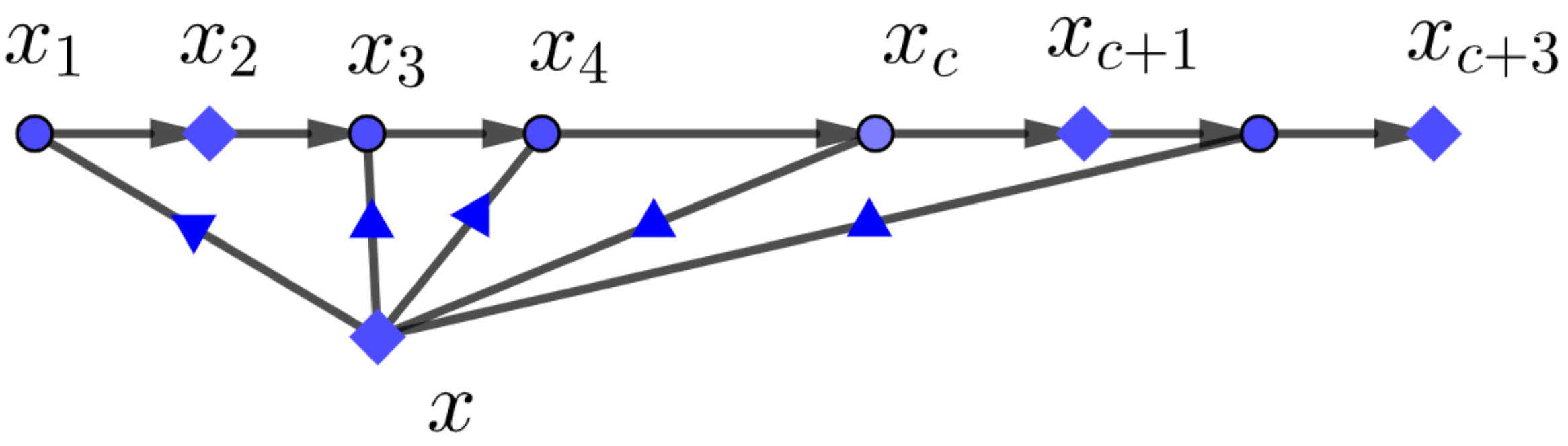}}
  \caption{The structure of $D[P\cup \{x\}]$ of Type I -- Type IV in Claim \ref{clm3}.}
  \label{clm3-1}
\end{figure}

\bigskip

Clearly, $x \in V_1$ or $x \in V_{c+1}$. Otherwise, there is a $(c+2)$-cycle $xx_1Px_{c+1}x$ in $D$. Suppose that $x \in V_1$. If $x_3 \notin V_1$, then $xx_3Px_{c+3}x$ is an $(c+2)$-cycle because $x$ has at least two out-neighbours. Then $x_3 \in V_1$ and $x \rightarrow x_4$. Note that $x_4$ and $x_{c+5}$ belong to the same partite set. If $x_5\rightarrow x$, then $q=4$, $m\geq c+4$ and $P$ is isomorphic to Type I in Fig. \ref{clm3-1}. If $x\rightarrow x_5$, we obtain a $(c+2)$-cycle $xx_5Px_{c+5}x$ when $m \geq c+5$. Hence $m=c+3$ or $m=c+4$ and $P=x_1x_2\cdots x_{c+3}(x_{c+4})$ where $x_3,x_{c+2} ( x_{c+4}) \in V_1$. Next, suppose that $x \in V_{c+1}$. Obviously, there is a $(c+2)$-cycle $xx_2Px_{c+2}x$ when $x \notin V_2$. Similarly, when $x_4 \rightarrow x$ we have $q=3$ and $m=c+2$, which is impossible. Then $x \rightarrow x_4$. We obtain a $(c+2)$-cycle $xx_4Px_{c+4}x$ when $m \geq c+4$. Hence $m=n+3$ and $P=x_1x_2\cdots x_{c+3}$ where $x_{c+1} \in V_2$. In a word, when $x$ has at least two in-neighbours and two out-neighbours, $P$ has four specific structures as described in Fig. \ref{clm3-1} based on the partite set which $x$ belongs to.

Second, we suppose that $x$ has either one in-neighbour or one out-neighbour. Then (i) $x \in V_1$ or $x \in V_m$ and $x$ has the same in-neighbours and out-neighbours on $P$ as $x_1$ or $x_m$; or (ii) $x\rightarrow x_1$, $x_2Px_m \Rightarrow x$ and $x \in V_{c+1}$; or (iii) $x_m\rightarrow x$, $x \Rightarrow x_1Px_{m-1}$ and $x \in V_{m-c}$.
\end{proof}

By Claim \ref{clm2} and Claim \ref{clm3}, we get the following.

\begin{Proposition}\label{property1} If $x \in N^+(P) \cap N^-(P)$, then $x$ and $P$ satisfy one of the following statements.

(i) $x$ and one of $\{x_1,x_2,x_{m-1},x_m\}$ belong to the same partite set and their in-neighbours and out-neighbours on $P$ are same;

(ii) $x$ and $x_l \in \{x_3,x_4,x_{m-3},x_{m-2}\}$ belong to the same partite set and their in-neighbours and out-neighbours on $P$ are same, where $x_3 \in V_1$ when $l=3$; $x_4 \in V_1$ when $l=4$; $x_{m-2} \in V_m$ when $l=m-2$; and $ x_{m-3} \in V_m$ when $l=m-2$;

(iii) $D[P \cup x]$ has four specific structures Type(I--IV) which are shown in Fig. \ref{clm3-1};

(iv) $x\rightarrow x_1$, $x_2Px_m \Rightarrow x$ and $x \in V_{c+1}$ or $x_m\rightarrow x$, $x \Rightarrow x_1Px_{m-1}$ and $x \in V_{m-c}$.
\end{Proposition}

\medskip

\begin{figure}[h]
  \centering
  \includegraphics[width=0.9\textwidth]{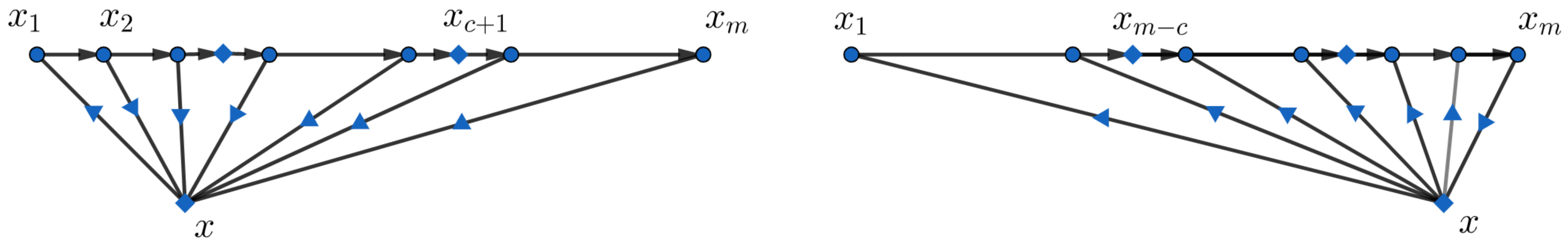}
  \caption{The structure of $D[P\cup \{x\}]$ of Proposition \ref{property1}(iv).}
  \label{clm3-3}
\end{figure}

\medskip

Next, suppose that $x$ has only in-neighbours in $V(P)$, i.e., $P \Rightarrow x$. Since $D$ is strong, there is a path from $x$ to $P$. Let $P^\prime =x \cdots x^\prime x^{\prime\prime}$ be a shortest path such that $x^{\prime\prime} \in N^-(P)$. If $N^-(x^{\prime\prime}) \cap V(P) = \emptyset$, then there is an integer $j \leq 4$ such that $x_{j+c-1} \rightarrow x^\prime$ and further $D$ contains a $(c+2)$-cycle $x^\prime x^{\prime\prime} x_jPx_{j+c-1}x^\prime$, a contradiction. Then $x^{\prime\prime} \in N^-(P)\cap N^+(P)$. By Proposition \ref{property1}, there are several all possible structures of $D[x^{\prime\prime}\cup P]$. In each case, we obtain a $(c+2)$-cycle or $diam(D) \geq m$, which contradicts the initial assumption or Claim \ref{clmc+2}.

\begin{enumerate}[\text{Case} 1:]
  \item $x^{\prime\prime}$ satisfies Proposition \ref{property1} (i) and $x_l\in \{x_2,x_{m-1},x_m\}$. It is easy to check that $D$ contains a $(c+2)$-cycle.

  \item $x^{\prime\prime}$ satisfies Proposition \ref{property1} (ii). There exists a $(c+2)$-cycle $x_2x_3x^\prime x^{\prime\prime} $ $x_4 P x_cx_2$ when $l=3$ (or $x_2x_3x_4x^\prime x^{\prime\prime} x_5Px_cx_2$ when $l=4$, resp.). For the case $l=m-2$ and $l=m-3$, we can obtain a $(c+2)$-cycle similarly.

  \item $x^{\prime\prime}$ satisfies Proposition \ref{property1} (iii). $D[P\cup \{x^\prime,x^{\prime\prime}\}]$ contains a $(c+2)$-cycle $x_1x_2x^\prime x^{\prime\prime}x_3Px_cx_1$ (or $x_1x_2x_3x^\prime x^{\prime\prime}x_4Px_cx_1$) for Type I, III. For Type II, there is a $(c+2)$-cycle $x_{m-1}x^{\prime}x^{\prime\prime}x_{m-c}Px_{m-1}$ unless there is no arc between $x^{\prime\prime}$ and $x_{m-c}$. Moreover $D$ contains $x_mx^{\prime}x^{\prime\prime}x_{m-c+1}Px_m$ unless there is no arc between $x^{\prime}$ and $x_m$. Then $x_{m-2}x^{\prime}x^{\prime\prime}x_{m-c-1}Px_{m-2}$ is a cycle when $x^{\prime\prime} \nrightarrow x_{m-c}$ and $x_m \nrightarrow x^\prime$.

  \item $x^{\prime\prime}$ satisfies Proposition \ref{property1} (iv) and $x_m \rightarrow x^{\prime\prime}$. There is a $(c+2)$-cycle \begin{math}x_{m-c} x^\prime x^{\prime\prime} x_{m-1}x_{m-c}\end{math} unless $x_{m-c}$, $x^{\prime\prime}$ and $x_{m-1}$ belong to the same partite set. Then $D[P\cup \{x^\prime,x^{\prime\prime}\}]$ contains a $(c+2)$-cycle $x_{m-c} x^\prime x^{\prime\prime} x_{m-c+2} P x_{m} x_{m-c}$.

  \item $x^{\prime\prime}$ satisfies Proposition \ref{property1} (i) $x_l\in \{x_2,x_{m-1},x_m\}$ or (iv) $x^{\prime\prime} \rightarrow x_1$. According to the analysis of Cases 1 -- 4, we get $dist(x^\prime,x_m) \geq m$, a contradiction.
\end{enumerate}

Hence, it is impossible that $x$ has only in-neighbours on $P$. Analogously, we can show that $D-P$ does not have any vertex which only has out-neighbours on $P$.

Since each partite set of $D$ has at least two vertices, $P$ is not of Type III or and Type IV $m\geq 2c+1$. In the following, we show that no vertex out of $P$ satisfies (iv). Assume that there is a vertex $x$ satisfying (iv) and $x\rightarrow x_1$, $x_2Px_m \Rightarrow x$. If there is a vertex $y$ out of $P$ such that $ x \rightarrow y$, it is easy to obtain that $y$ and $x_1$ have the same in-neighbours and out-neighbours on $P$; or $y$ satisfies (iv) and $x_m\rightarrow y$, $y \Rightarrow x_1Px_{m-1}$; or $D[P \cup y]$ is of Type II. Thus $x_cxyx_{c+1}x_2Px_c$ is a $(c+2)$-cycle unless $x_{c+1} \nrightarrow x_2$. However, we obtain that $D$ contains $x_3xyx_4Px_cx_1x_2x_3$ or $x_3xyx_5Px_{c+1}x_1x_2x_3$. Thus $y$ and $x_1$ have the same adjacency to $P$. This implies that $dist(x,x_m)=m$, a contradiction. Analogously, if there is a vertex $x$ satisfying (iv) and $x_m\rightarrow x$, $x \Rightarrow x_1Px_{m-1}$, we will get $dist(x_1,x)=m$, a contradiction. Hence no vertex out of $P$ satisfies (iv). Finally, if there exist vertices $x$ of Type I and $y$ of Type II such that $x\rightarrow y$, then $D$ contains a $(c+2)$-cycle $x_1Px_4xyx_5Px_cx_1$ or $x_1Px_4x_6Px_{c+1}x_1$, a contradiction. Thus for any vertex $x$ of Type I and any vertex $y$ of Type II, there is no arc between $x$ and $y$ or $y \rightarrow x$. Observe that $D$ is isomorphic to a member of $\mathcal{Q}_m$. This proves Theorem \ref{main}.$\Box$

\section{Data Availability Statement}
No data were generated or used during the study.

\nocite{*}
\bibliographystyle{abbrvnat}
\bibliography{cite}

\begin{thebibliography}{10}
\providecommand{\natexlab}[1]{#1}
\providecommand{\url}[1]{\texttt{#1}}
\expandafter\ifx\csname urlstyle\endcsname\relax
  \providecommand{\doi}[1]{doi: #1}\else
  \providecommand{\doi}{doi: \begingroup \urlstyle{rm}\Url}\fi

\bibitem[Balakrishnan and Paulraja(1984)]{Balakrishnan}
R.~Balakrishnan and P.~Paulraja.
\newblock Note on the existence of directed {$(k+1)$}-cycles in diconnected
  complete {$k$}-partite digraphs.
\newblock \emph{J. Graph Theory}, 8\penalty0 (3):\penalty0 423--426, 1984.

\bibitem[Bang-Jensen and Gutin(1998)]{Bang-Jensen3}
J.~Bang-Jensen and G.~Gutin.
\newblock Generalizations of tournaments: a survey.
\newblock \emph{J. Graph Theory}, 28\penalty0 (4):\penalty0 171--202, 1998.

\bibitem[Bang-Jensen and Gutin(2001)]{Bang-Jensen}
J.~Bang-Jensen and G.~Gutin.
\newblock \emph{Digraphs: {T}heory, algorithms and applications}.
\newblock Springer Monographs in Mathematics. Springer-Verlag London, Ltd.,
  London, 2001.

\bibitem[Bang-Jensen and Gutin(2018)]{Bang-Jensen2}
J.~Bang-Jensen and G.~Gutin.
\newblock \emph{Classes of directed graphs}.
\newblock Springer Monographs in Mathematics. Springer, Cham, 2018.

\bibitem[Beineke and Reid(1978)]{Beineke}
L.~W. Beineke and K.~B. Reid.
\newblock \emph{Tournaments, Selected topics in graph theory}.
\newblock Academic Press, Inc., London-New York, 1978.

\bibitem[Bondy(1976)]{bondy}
J.~A. Bondy.
\newblock Diconnected orientations and a conjecture of {L}as {V}ergnas.
\newblock \emph{J. London Math. Soc. (2)}, 14\penalty0 (2):\penalty0 277--282,
  1976.

\bibitem[Guo and Volkmann(1996)]{guo}
Y.~Guo and L.~Volkmann.
\newblock A complete solution of a problem of {B}ondy concerning multipartite
  tournaments.
\newblock \emph{J. Combin. Theory Ser. B}, 66\penalty0 (1):\penalty0 140--145,
  1996.

\bibitem[Gutin(1982)]{gutin1}
G.~Gutin.
\newblock On cycles in complete {$n$}-partite digraphs.
\newblock \emph{Depon. in VINITI, No. 2473, Gomel Politechnic Institute}, 1982.

\bibitem[Gutin(1984)]{gutin2}
G.~Gutin.
\newblock Cycles in strong {$n$}-partite tournaments.
\newblock \emph{Vests\={\i} Akad. Navuk BSSR Ser. F\={\i}z.-Mat. Navuk},
  \penalty0 (5):\penalty0 105--106, 1984.

\bibitem[Volkmann(2007)]{Volkmann}
L.~Volkmann.
\newblock Multipartite tournaments: a survey.
\newblock \emph{Discrete Math.}, 307\penalty0 (24):\penalty0 3097--3129, 2007.

\end{thebibliography}

\end{document}